\documentclass[a4paper,12pt,draft]{article}

\usepackage[top=1in,left=1in,right=1in,bottom=1.5in]{geometry}

\usepackage[T1]{fontenc}
\usepackage{amsmath,amssymb,amsfonts,amsthm,mathrsfs,eucal,dsfont}

\theoremstyle{plain}
\newtheorem{theorem}{Theorem}
\newtheorem{lemma}[theorem]{Lemma}
\newtheorem{proposition}[theorem]{Proposition}
\newtheorem{corollary}[theorem]{Corollary}
\theoremstyle{definition}
\newtheorem{remark}[theorem]{Remark}
\newtheorem{definition}[theorem]{Definition}
\newtheorem{example}[theorem]{Example}
\newtheorem*{ack}{Acknowledgement}

\newcommand{\nat}{\mathds{N}}
\newcommand{\real}{\mathds{R}}
\newcommand{\rn}{{\mathds{R}^n}}
\newcommand{\comp}{\mathds{C}}
\newcommand{\Ee}{\mathds{E}}
\newcommand{\I}{\mathds{1}}
\newcommand{\form}{\mathcal{E}}

\newcommand{\nnorm}[1]{\|#1\|}
\newcommand{\scalar}[1]{\langle#1\rangle}
\newcommand{\supp}{\operatorname{supp}}
\newcommand{\id}{\operatorname{id}}

\renewcommand{\leq}{\leqslant}

\renewcommand{\geq}{\geqslant}

\renewcommand{\Re}{\operatorname{Re}}
\renewcommand{\Im}{\operatorname{Im}}

\title{\bfseries Transition density estimates for a class of L\'evy and L\'evy-type processes}

\author{%
    \textsc{Viktorya Knopova}%
    \thanks{V.M.\ Glushkov Institute of Cybernetics,
            NAS of Ukraine,
            40, Acad.\ Glushkov Ave.,
            03187, Kiev, Ukraine,
            \texttt{vic\underline{ }knopova@gmx.de}}
    \textrm{\ \ and\ \ }
    \stepcounter{footnote}\stepcounter{footnote}\stepcounter{footnote}
    \stepcounter{footnote}\stepcounter{footnote}%
    \textsc{Ren\'e L.\ Schilling}%
    \thanks{Institut f\"ur Mathematische Stochastik,
              Technische Universit\"at Dresden,
              D-01062 Dresden, Germany,
              \texttt{rene.schilling@tu-dresden.de}}
    }

\date{}

\begin{document}
\maketitle

\begin{abstract}
    \noindent
    We show  on-and off-diagonal upper estimates for the transition densities of
    symmetric  L\'evy and L\'evy-type processes. To get the  on-diagonal estimates we prove a Nash
    type inequality for the related Dirichlet form. For the off-diagonal estimates  we assume that
    the characteristic
    function of a L\'evy (type) process is analytic, which allows to apply the complex analysis
    technique.


\end{abstract}

\section{Introduction}

Transition density estimates for jump processes received much
attention during the recent years. Two-sided heat kernel estimates
for a class of stable-like processes in $\rn$ were obtained by
Kolokoltsov \cite{Kol00}; Bass and Levin \cite{BL02} used a
completely different approach to get transition density estimates
for discrete time Markov chains in $\mathbb{Z}^d$ with certain
conductance. For the transition density estimates for a tempered
stable process see \cite{S10} and the references therein. Chen and
Kumagai  \cite{CK03} obtained two-sided heat kernel estimates and
a parabolic Harnack inequality on $d$-sets, and further extended
these results to symmetric jump-type processes on metric measure
spaces \cite{CK08}. Further results on two-sided heat kernel
estimates and a version of the parabolic Harnack inequality for
symmetric jump processes are contained in \cite{BBCK} and
\cite{BGK}. Pure jump processes whose jump kernel is comparable to
the one of truncated stable-like process are considered by Chen,
Kim and Kumagai  \cite{CKK07}.  See \cite{CKK10} for the heat
kernel estimates for a class of symmetric jump processes in $\rn$
with exponentially decaying jump kernels.

The techniques of getting upper bounds are often based on the paper \cite{CKS} by Carlen, Kusuoka and Stroock where on- and off-diagonal upper bounds for a class of symmetric Markov processes are obtained. The off-diagonal upper bound is obtained in \cite{CKS} in terms of the so-called carr\'e du champ operator which is uniquely determined by  the Dirichlet form related to the Markov process. For L\'evy processes the expression of the carr\'e du champ operator is hidden in the representation of the transition density provided the latter exists. Such observations suggest that in the case of L\'evy processes  the off-diagonal upper bound can be obtained in a similar form as it was done in \cite{CKS}.

Our goal is to get the upper on- and off-diagonal transition density estimates for certain classes of L\'evy and L\'evy-type processes. In Section~\ref{sec1} we prove the equivalence of the Nash type inequality to the existence of the transition density of the related semigroup of probability measures, and find an on-diagonal upper bound for the density. Then we apply these results to study the estimates for the transition density of L\'evy processes whose characteristic exponent is of the form $f(|\xi|^2)$, where $f$ is a (complete) Bernstein function. The main result, Theorem~\ref{pr}, is contained in Section~\ref{sec2}: if the L\'evy process has no diffusion part and if it has exponential moments, we can use the complex analysis technique to get the off-diagonal upper estimates for the transition density provided it exists.

As a by-product we show that the function which controls the off-diagonal behaviour of the transition density is exactly the rate function of the L\'evy process, coming from the theory of large deviations. As an application in Section~\ref{sec3} we show that under the conditions of Theorem~\ref{pr} the transition density satisfies the large deviation principle.

In Section~\ref{sec4} we show how the results obtained in the previous sections can be generalized to a large class of L\'evy-type processes.

\bigskip\noindent
\textbf{Notation.}
$\mathbb{N}$, $\real$, $\comp$ are the sets of positive integers, real numbers and  complex numbers, respectively; $\rn$ is the $n$-dimensional  Euclidean space; $C_0^\infty(\rn)$ are the test functions, $S(\rn)$ is the Schwartz space of rapidly decreasing functions on $\rn$. By $f \asymp g$ we indicate that there exist some positive constants $c_1$, $c_2$ such that $c_1f\leq g\leq c_2 f$. We write $B(x,r)$ for the open ball with centre $x$ and radius $r$.

\section{Preliminary results}\label{sec0}
In this section we will briefly recall some basic facts on continuous negative definite functions and Dirichlet forms which we will need later on.

An $n$-dimensional L\'evy process $X=(X_t)_{t\geq 0}$ is a stochastic process with values in $\rn$ with stationary and independent increments and stochastically continuous sample paths. It is well known that such processes are---up to indistinguishability---characterized by the characteristic function
$$
    \Ee e^{i\xi\cdot X_t} = e^{-t\psi(\xi)},\quad\xi\in\rn,\; t\geq 0,
$$
and the characteristic exponent $\psi:\rn\to\comp$. The exponent $\psi:\rn\to\comp$ is a \emph{continuous negative definite function} (in the sense of Schoenberg) which is equivalent to saying that $\psi$ enjoys a \emph{L\'evy-Khintchine representation},
\begin{equation}\label{pnu}
    \psi(\xi)
    = i\ell\cdot\xi + \frac 12\,\xi\cdot Q\xi + \int_{\rn\setminus\{0\}} \left(1- e^{i\xi\cdot y} + \frac{i\xi \cdot y}{1+|y|^2}\right) \nu(dy).
\end{equation}
Here $\ell\in\rn$, $Q\in\real^{n\times n}$ is a positive semi-definite matrix, and $\nu(dy)$ is the  L\'evy measure, i.e.\ a measure on $\rn\setminus\{0\}$ such that $\int_{\rn\setminus\{0\}}(1\wedge |y|^2)\, \nu(dy)<\infty$. The triplet $(\ell,q,\nu)$ uniquely characterizes $\psi$, hence $X$.

Note that negative definiteness of $\psi$ does not mean that $-\psi$ is positive definite. Note also, that continuous negative definite functions are always polynomially bounded:
$$
    |\psi(\xi)| \leq c_\psi\,(1+|\xi|^2),\qquad
    c_\psi = \sup_{|\eta|\leq 1}|\psi(\eta)|.
$$
A standard reference for these and further properties is the monograph \cite{J01}.

In this paper we will only consider real-valued $\psi$ with $Q\equiv 0$. Thus, \eqref{pnu} becomes
\begin{equation}\label{pnur}
    \psi(\xi)
    = \int_{\rn\setminus\{0\}} \big(1- \cos(\xi\cdot y)\big)\, \nu(dy)
\end{equation}
and the associated L\'evy process $X$ is symmetric in the sense that $\text{law}(X_t) = \text{law}(-X_t)$.

Let $\psi:\rn\to\real$ and write $\widehat u(\xi) = (2\pi)^{-n}\int u(x)\,e^{-ix\cdot\xi}\,dx$ for the Fourier transform. Then
\begin{equation}
    -\psi(D)u(x)
    =- \int_{\rn} e^{i x\cdot\xi}\, \psi(\xi)\,\widehat{u}(\xi)\, d\xi,\quad u\in C_0^\infty(\rn),
\end{equation}
is a pseudo-differential operator with symbol $\psi$. It is not
hard to see that $-\psi(D)$ coincides on the test functions
$C_0^\infty(\rn)$ with the infinitesimal generator $A$ of the
L\'evy process $X$. More precisely, if we denote by $P_t u(x) :=
\Ee u(X_t+x)$ the convolution semigroup associated with $X$,
$P_t|_{C_0^\infty(\rn)}$  can be extended to a strongly
continuous, symmetric sub-Markovian semigroup on each
$L_p(\rn,dx)$, $p\in [1,\infty)$, as well as on the space of
continuous functions vanishing at infinity: $C_\infty(\rn) :=
\overline{C_0^\infty(\rn)}^{\nnorm\cdot_\infty}$. In any case,
$C_0^\infty(\rn)$ is an operator core, and the closure of
$-\psi(D)|_{C_0^\infty(\rn)}$ is the $L_p$-generator. On
$C_0^\infty(\rn)$ it is also possible to express the semigroup
$(P_t)_{t\geq 0}$ as pseudo-differential operators,
\begin{equation}
    P_t u(x)= \int_{\rn} e^{i x\cdot\xi} \,e^{-t\psi(\xi)}\,\widehat{u}(\xi)\, d\xi,\quad u\in C_0^\infty(\rn),
\end{equation}
see e.g.\ \cite[Chapter 4]{J01}.

We are mostly interested in the $L_2$ setting. The domain of the $L_2$ generator is a concrete function space
\begin{equation}
    H^{\psi,2}(\rn) :=
    \overline{C_0^\infty(\rn)}^{\nnorm\cdot_{\psi,2}}
    \quad\text{where}\quad
    \nnorm u_{\psi,2}^2 = \nnorm{\psi(D)u}_{L_2}^2 + \nnorm u_{L_2}^2,
\end{equation}
see \cite[Chapter 4.1]{J01}. Denote by $\form^\psi(u,v)$ the Dirichlet form associated with $-\psi(D)$, see \cite[Example 4.7.28]{J01}. It is given by
\begin{equation}\label{eform1}\begin{aligned}
    \form^\psi(u,v)
    &=\int_{\rn}\psi(\xi)\widehat{u}(\xi)\overline{\widehat{v}}(\xi)\, d\xi\\
    &=\frac{1}{2}\int_{\rn}\int_{\rn}(u(x+y)-u(x))(v(x+y)-v(x))\,\nu(dy)\,dx.
\end{aligned}\end{equation}
    The domain $D(\form^\psi) = H^{\psi,1}(\rn)$ is the closure of  $C_0^\infty(\rn)$ with respect to the norm induced by the scalar product
$$
    \form^\psi_1(u,v) := \form^\psi(u,v) + \scalar{u,v}_{L_2}.
$$

In Section~\ref{sec2} we will need the notion of a \emph{carr\'e
du champ operator} to show the off-diagonal upper bound for the
transition density. Below we give the definition of the carr\'e du
champ operator in the sense of Bouleau and Hirsch for a general
Dirichlet form $(\form, D(\form))$ on  $L_2(\Omega, m)$, see
\cite[Definition 4.1.2]{BH}. As usual, $m$ is a positive
$\sigma$-finite Radon measure on a locally compact measurable
space $(\Omega, \mathcal{F})$.
\begin{definition}\label{d412}
    We say that a Dirichlet form $(\form,D(\form))$ admits a \textit{carr\'e du champ} if there exists a subspace $\mathcal{C}$ of $D(\form)\cap L_\infty(\Omega, m)$ which is dense in $D(\form)$ such that for all $f\in \mathcal{C}$ there exists a function $\tilde{f}$ such that
    $$
        2\form(fh,f)-\form(h,f^2)=\int h \tilde{f}\, dm
        \quad
        \text{for all $h\in D(\form)\cap L_\infty(\Omega, m)$}.
    $$
\end{definition}
If the assumptions of Definition \ref{d412} are satisfied, then there exists a unique positive symmetric and continuous bilinear operator  $\Gamma: D(\form)\times D(\form)\to L_1(\Omega, m)$, called the \emph{carr\'e du champ operator}, such that  for all $f,g,h\in D(\form)\cap L_\infty(\Omega, m)$
\begin{equation}\label{cdc1}
    \form(fh,g)+\form(gh,f)-\form(h,fg)=\int h \,\Gamma(f,g)\, dm
\end{equation}
cf.\  \cite[Proposition 4.1.3]{BH}. For a L\'evy process we have an explicit  representation of the carr\'e du champ operator
\begin{equation}\label{Gam1}
    \Gamma(u,u)=\Gamma^\psi(u,u)=\frac{1}{2}\int_{\rn} (u(x+y)-u(x))^2 \,\nu(dy)
\end{equation}
where $\nu$ is the L\'evy measure and $\psi$ is given by \eqref{pnur}.

\medskip
We are particularly interested in radially symmetric negative
definite functions $\psi:\rn\to\real$. In this case $\psi(\xi) =
f(|\xi|^2)$ for some function $f:[0,\infty)\to[0,\infty)$. In
order that $\xi\mapsto f(|\xi|^2)$, $\xi\in\rn$, is negative
definite for \emph{every} dimension $n\geq 1$ it is necessary and
sufficient that $f$ is a Bernstein function. This is a classic
result by Sch\"onberg and Bochner, see \cite{ssv} or \cite{JS}.
\begin{definition}
\begin{enumerate}
\item[a)]
    A real-valued function  on $[0,\infty)$  is called a \emph{Bernstein function}, if
    $$
        f(t) = a+bt + \int_{(0,\infty)} (1-e^{-st})\,\mu(dt)
    $$
    with $a,b\geq 0$ and a measure $\mu$ on $(0,\infty)$ satisfying $\int_{(0,\infty)} (1\wedge t)\,\mu(dt)<\infty$.
\item[b)]
    A Bernstein function is said to be a \emph{complete Bernstein function} if the representing measure $\mu(dt)$ is of the form $m(t)\,dt$ with a completely monotone density $m(t)$, i.e.\ where $m(t)$ is the Laplace transform of a positive measure on $[0,\infty)$.
\end{enumerate}
\end{definition}
Bernstein functions are equivalently characterized by the requirement that
$$
     f\in C^\infty(0,\infty), \quad f\geq 0, \quad\text{and}\quad (-1)^k f^{(k)}(x)\leq 0\quad\text{for all}\quad k\geq 1.
$$
For the properties of Bernstein and complete Bernstein functions we refer to \cite[Chapter~3.9]{J01} and \cite{ssv}. For our purposes we will only mention the following estimate
\begin{equation}\label{derf}
    \big| f^{(k)}(x)\big|\leq \frac{k!}{x^k}\,f(x),\quad k\geq 0,\; x >0.
\end{equation}

\section{An on-diagonal estimate for the transition density}\label{sec1}

In this section we will prove  the equivalence of the Nash type inequality for general Dirichlet forms and the absolute continuity of the related probability measure with respect to Lebesgue measure on $\rn$; then we will establish an upper bound for this transition density. For the special case where $f(x)=x^\alpha$, $0<\alpha<1$, we refer to \cite{CKS} and to the discussion in \cite{ben-mah}. The importance of Nash and Sobolev type inequalities for (symmetric) Markovian semigroups has been pointed out in the monograph \cite{vcs}.

Let $f$ be a Bernstein function, $(\form,D(\form))$ be a symmetric
regular Dirichlet form with generator $(A,D(A))$ and associated
$L_2(\rn)$--  sub-Markovian semigroup $(P_t)_{t\geq 0}$. Since the
operators $P_t$ are symmetric, we may (and do) extend $P_t$ to the
whole scale $L_p$, see e.g.\ \cite{Da}. By $\|P_t\|_{L_p\to L_q}$
we indicate the corresponding $L_p$--$L_q$ operator norm.

\begin{proposition}\label{nash2}
    Let  $(P_t)_{t\geq 0}$ be the semigroup  associated with the regular symmetric
    Di\-rich\-let form $(\form,D(\form))$ and let $f:(0,\infty)\to [0,\infty)$ be some
    differentiable function such that
     $f$ is subadditive,
    $$
        f'(x) > 0
        \text{\ \ and\ \ }
        f'(x) \leq \frac 1x\,f(x)
        \quad\text{for all\ \ }x>0.
    $$
    Then the following Nash inequality
\begin{equation}\label{n222}
    \|u\|_{L_2}^2 \, f\left(\left[\frac{\|u\|_{L_2}}{\|u\|_{L_1}}\right]^{4/n}\right)
    \leq
    C_0\left[\form(u,u)+\delta\|u\|_{L_2}^2\right],
\end{equation}
    ($C_0>0$ and $\delta\geq 0$ are some constants) holds if, and only if,
\begin{equation}\label{on-diag}
    \|P_t\|_{L_1\to L_\infty}
    \leq \left[ f^{-1} \left(\frac{1}{ \gamma t}\right) \right]^{n/2}e^{ 2\delta t}.
\end{equation}
     where the constant $\gamma>0$ depends only on $C_0$, and $\delta$ is as in \eqref{n222}.
\end{proposition}
\begin{proof}
Assume that \eqref{n222} holds. Fix some $u\in D(\form)\cap L_\infty(\rn)\cap L_1^+(\rn)$ with $\|u\|_{L_1}=1$ and set $u_t(x):=P_tu(x)$ and $\phi(t):=e^{-2\delta t} \|u_t\|_{L_2}^2$. Since $\form(u,u)=\nnorm{(-A)^{1/2}u}^2_{L_2}$ and $\frac d{dt}P_t = AP_t$ we find
\begin{align*}
    -\frac{d \phi(t)}{dt}
    =2e^{-2\delta t} \big[ \form(u_t,u_t)+\delta\|u_t\|_{L_2}^2\big]
    &\geq e^{-2\delta t}  \frac{2}{C_0} \|u_t\|_{L_2}^2\,f\big(\|u_t\|_{L_2}^{4/n}\big)\\
    &\geq  \frac{2}{C_0} \,\phi(t)\,f\big(\phi(t)^{2/n}\big).
\end{align*}
For the last estimate we used the fact that $f(e^{4\delta t/n}\phi^2(t))\geq f(\phi^2(t))$. It is not hard to see that for each $n\in\nat$ the function $f_1(x):= f(x^{1/n})$ enjoys the same properties as $f$: it is strictly increasing and
$$
    f_1'(x)
    = \frac 1n\, f'(x^{1/n}) \, x^{1/n - 1}
    \leq \frac 1n\, f(x^{1/n})\, x^{-1/n}\,x^{1/n - 1}
    \leq \frac 1x\, f_1(x).
$$
Combining this with the above estimate for $\phi'$ we get
\begin{equation}\label{f-der}
    \left(\frac{1}{f_1(\phi^2)}\right)'
    =-\frac{2\,f'_1(\phi^2)\,\phi\,\phi'}{f_1^2(\phi^2)}
    \geq -\frac{2 \,\phi'}{\phi\, f_1(\phi^2)}
    \geq  \frac{ 4 \, \phi \,f_1 (\phi^2)}{ C_0\, \phi\, f_1(\phi^2)}
    = \frac{4}{C_0} .
\end{equation}
Now we integrate \eqref{f-der} from $0$ to $t$ and conclude
$$
    \frac{1}{f_1(\phi^2(t))}
    \geq \frac{1}{f_1(\phi^2(t))} - \frac{1}{f_1(\|u\|_{L_2}^{4/n})}
    \geq  \frac{4}{C_0}\,t
$$
or
$$
    \phi(t)\leq  \left[f^{-1}\left(\frac{1}{\gamma  t}\right) \right]^{n/2},\quad \gamma = \frac{4}{C_0}
$$
This implies that
$$
    \|P_t\|_{L_1\to L_2}
    \leq \left[f^{-1}\left(\frac{1}{\gamma t}\right) \right]^{n/4} e^{\delta t},
$$
and, by the usual semigroup and symmetry arguments, see e.g.\ \cite{vcs}, we finally get
$$
    \|P_t\|_{L_1\to L_\infty}
    \leq \|P_{t/2}\|_{L_1\to L_2} \cdot \|P_{t/2}\|_{L_2\to L_\infty}
    = \|P_{t/2}\|_{L_1\to L_2}^2
    \leq \left[f^{-1}\left(\frac{1}{\gamma t}\right) \right]^{n/2} e^{2 \delta t}
$$
which is \eqref{on-diag}  with $\gamma = 4/C_0$.

\medskip
Conversely, assume that \eqref{on-diag} holds. Let  $v\in D(A)\cap L_\infty(\rn) \cap L_1^+(\rn)$ and set $v_t(x):=e^{-\delta t}\,P_tv(x)$. Then
$\|v_t\|_{L_\infty} \leq \|v\|_{L_1} [f^{-1} (1/\gamma t) ]^{n/2}$, and
$$
    v_t=v-\int_0^t (\delta\id +A) v_s \, ds.
$$
Since $P_t$ and $A$ are self-adjoint operators,
$$
    \scalar{v,P_tv}_{L_2}\leq \|v\|_{L_2}^2
    \quad\text{and}\quad
    \scalar{v,Av_t}_{L_2}
    =e^{-\delta t} \scalar{v,A P_t v}_{L_2}
    = e^{-\delta t}  \|(AP_t)^{1/2} v\|_{L_2}^2
    \leq \form(v,v).
$$
Therefore
\begin{align*}
    \|v\|_{L_1}^2\,  \left[f^{-1}\left(\frac{1}{\gamma  t}\right) \right]^{n/2}
    \geq \|v\|_{L_1}\|v_t\|_{L_\infty}
    \geq \scalar{v,v_t}_{L_2}
    &=\|v\|_{L_2}^2 -\int_0^t \scalar{v,(\delta\id +A)v_s}_{L_2}\,ds\\
    &\geq \|v\|_{L_2}^2 -t\,\left[\form(v,v)+\delta \|v\|_{L_2}^2\right],
\end{align*}
and we arrive at
$$
    \delta \|v\|_{L_2}^2 +\form(v,v)\geq \frac{\|v\|_{L_2}^2}{t}\left[ 1-\frac{\|v\|_{L_1}^2}{\|v\|_{L_2}^2}\left[f^{-1}
    \left(\frac{1}{ \gamma t}\right) \right]^{n/2}\right].
$$
Choosing $t=t_0=1\Big/ \gamma  f\Big(\left(\frac{\|v\|_{L_2}}{\sqrt{2}\,\|v\|_{L_1}}\right)^{4/n}\Big)$, we get
$$
    \nnorm u_{L_2}^2\, f\left(\left[\frac{\|v\|_{L_2}}{\sqrt{2}\,\|v\|_{L_1}}\right]^{4/n}\right)
    \leq \frac 2\gamma\,\left[\form(u,u)+\delta\nnorm u_{L_2}^2\right].
$$
Since $f$ is subadditive, we see that $\frac 14 f(y) \leq f(y/4) \leq f(y/4^{1/n})$ and this yields \eqref{n222}  with $C_0 = 8/\gamma$.
\end{proof}

We can apply Proposition~\ref{nash2} to get on-diagonal estimates for a wide class of L\'evy processes which are subordinate to Brownian motion. It is well-known that the symbols of such L\'evy processes are of the form $\xi\mapsto f(|\xi|^2)$ where $f(x)$ is a Bernstein function. The corresponding Dirichlet form can be expressed as
\begin{equation*}
    \form^{f(|\cdot|^2)}(u,u)
    = \int_\rn f(|\xi|^2)\,|\widehat u(\xi)|^2\,d\xi,\quad u\in C_0^\infty(\rn).
\end{equation*}
From a function-space point of view this representation is rather complicated and it is desirable to get an equivalent representation involving differences of the function $u$, cf.\ \cite{JS}. Let us, for simplicity, assume that $f(x)$ has no linear term. Under the additional condition that there exists some $0<\kappa<1$ such that
\begin{equation*}
    t\mapsto f(t)\,t^{-\kappa}\quad \text{is increasing as}\quad t\to \infty
\end{equation*}
it was shown in \cite{KZ} and \cite{Mo} that
$$
    \form^{\psi_1}(u,u)
    =\frac{1}{2}\int_{\rn}\int_{B(0,1)} |u(x)-u(y)|^2\, f\left(\frac{1}{|x-y|^2}\right) \frac{dy\, dx}{|x-y|^n},
$$
where
\begin{equation}\label{psi1}
    \psi_1(\xi):=\int_{B(0,1)} (1-\cos(\xi\cdot y))f\left(\frac{1}{|y|^2}\right)\, \frac{dy}{|y|^n},
\end{equation}
is a negative definite function, defines an equivalent Dirichlet form. This is to say that $D(\form^{f(|\cdot|^2)})=D(\form^{\psi_1})$ and that
$$
    \form^{f(|\cdot|^2)}(u,u) + \scalar{u,u}_{L_2}
    \asymp
    \form^{\psi_1}(u,u) + \scalar{u,u}_{L_2}
$$
for all $u\in D(\form^{f(|\cdot|^2)})=D(\form^{\psi_1})$.

Obviously, if  the Nash inequality \eqref{n222} holds for $\form^{\psi_1}$, it also holds for $\form^{f(|\cdot|^2)}$ and vice versa. Note that the Bernstein function $f$ satisfies the assumptions of Proposition \ref{nash2}.

\begin{lemma}\label{nash}
    Let $f$ be a Bernstein function without linear term such that for some $\kappa\in (0,1)$ the function $t\mapsto f(t)\,t^{-\kappa}$ increases as $t\to\infty$.  Then the transition density of the L\'evy process with L\'evy exponent $f(|\xi|^2)$ (or $\psi_1(\xi)$) exists and  there exist suitable constants $c>0, \gamma\in (0,1]$ such that
\begin{equation}\label{pf1}
    p_t(x)
    \leq c \left[ f^{-1}\left(\frac{1}{\gamma t}\right)\right]^{n/2},
    \quad\text{for all\ \ $0<t\leq 1$ and $x\in\rn$}.
\end{equation}
\end{lemma}
\begin{proof}
We check \eqref{n222} for the Dirichlet form $\form^{\psi_1}$. Let $u\in L_\infty(\rn)\cap L_1(\rn)$ and fix some $0<r<1$.  Then
\begin{align*}
    \form^{\psi_1}(u,u)
    &\geq \int_{\rn}\int_{B(x,r)}\frac{|u(x)-u(y)|^2 }{|x-y|^n} \, f\left( \frac{1}{|x-y|^2}\right) dy\, dx\\
    &\geq f\left(\frac{1}{r^2}\right)\frac{1}{r^n} \int_{\rn}\int_{B(x,r)} |u(x)-u(y)|^2 \, dy\, dx\\
    &\geq c\,f\left( \frac{1}{r^2}\right)\frac{1}{r^{2n}} \int_{\rn}\left| \int_{B(x,r)} (u(x)-u(y)) \, dy \right|^2 dx\\
    &\geq  c'\, f\left( \frac{1}{r^2}\right)\int_{\rn} |u(x)- u_r|^2  \, dx\\
    &=c'  f\left( \frac{1}{r^2}\right) \|u(\cdot)-  u_r\|^2_{L_2},
\end{align*}
where $u_r := \tau_n^{-1}r^{-n}\int_{B(x,r)} u(y)\, dy$ and  $\tau_n = \pi^{\frac{n}{2}} / \Gamma(1+\frac{n}{2})$ is the volume of the unit ball $B(0,1)$ in $\rn$. Now observe that
\begin{gather*}
    \|u_r\|_{L_1}
    \leq \int_\rn \frac 1{\tau_n\,r^n} \int_{B(x,r)} |u(y)|\, dy\,dx
    = \frac 1{\tau_n\,r^n} \int_{B(0,r)} \int_\rn |u(x+y)|\,dx\,dy
    = \|u\|_{L_1}
\end{gather*}
and that $\|u_r\|_{L_\infty} \leq \tau_n^{-1} r^{-n}\|u\|_{L_1}$. By H\"{o}lder's inequality we get
\begin{equation*}
    \|u_r\|_{L_2}^2
    \leq \|u_r \|_{L_\infty} \cdot \| u_r\|_{L_1}
    \leq \tau_n^{-1}\,\frac{\|u\|_{L_1}^2}{r^n},
\end{equation*}
and so
\begin{align*}
    \|u\|_{L_2}^2
    =\|u-u_r+u_r \|_{L_2}^2
    \leq 2\,\big(\|u-u_r\|_{L_2}^2 +\|u_r\|_{L_2}^2\big)
    \leq 2\,\|u-u_r\|_{L_2}^2 +2\tau_n^{-1}\, \frac{\|u\|_{L_1}^2}{r^{n}}.
\end{align*}
Thus, we arrive at
\begin{equation}\label{n3}
    \|u\|_{L_2}^2
    \leq \frac{c}{f(1/r^2)}\,\form^{\psi_1}(u,u)+c\,r^{-n}\,\|u\|_{L_1}^2.
\end{equation}
We will now distinguish between two cases. Let $h(r) := r^n / f(r^{-2})$.

\medskip
\emph{Case 1.}  Suppose that $\|u\|_{L_1}^2 < h\big(1)\,\form^{\psi_1}(u,u)$. Since
$$
    \lim_{r\to 0} h(r) = \lim_{x\to\infty} (x^n\,f(x^2))^{-1}=0,
$$
the equation $h(r) = \|u\|_{L_1}^2 /\form^{\psi_1}(u,u)$ has a solution $r_0\in (0,1)$. Substituting $\form^{\psi_1}(u,u)=\|u\|_{L_1}^2 / h(r_0)$ into \eqref{n3} with $r\equiv r_0$  yields
$$
    \|u\|_{L_2}^2
    \leq 2c\,r_0^{-n}\,\|u\|_{L_1}^2
    = \frac{\|u\|_{L_1}^2 }{\left[h^{-1}\left(\dfrac{\|u\|_{L_1}^2}{\form^{\psi_1}(u,u)}\right)\right]^n}.
$$
A few elementary rearrangements give
$$
    \form^{\psi_1}(u,u)
    \geq \frac{\|u\|_{L_1}^2}{h\left(\Big(\frac{\|u\|_{L_1}}{\|u\|_{L_2}}\Big)^{2/n}\right)}
$$
which becomes \eqref{n222} if we express $h$ in terms of $f$.

\medskip
\emph{Case 2.} $\|u\|_{L_1}^2 \geq h(1) \,\form^{\psi_1}(u,u)$ where $c$ is the constant appearing in \eqref{n3}. Take $r=1$ in \eqref{n3}. Then
$$
    \|u\|_{L_2}^2\leq 2c\, \|u\|_{L_1}^2,
$$
and by monotonicity $f\Big( \Big(\frac{\|u\|_{L_2}^2}{\|u\|_{L_1}^2}\Big)^{4/n}\Big) \leq C$. This implies \eqref{n222}, and by Proposition \ref{nash2} we get \eqref{pf1}.
\end{proof}

\bigskip
There is a related situation where we can apply Proposition \ref{nash2}. Consider a L\'evy process with the symbol
\begin{equation*}
    \psi_\infty(\xi):=\int_{\rn} (1-\cos(\xi\cdot y))\,g\left(\frac{1}{|y|^2}\right) \frac{dy}{|y|^n}
\end{equation*}
where $g$ is a complete Bernstein function. By \cite[Theorem 3.5]{JS}, the corresponding Dirichlet form $\form^{\psi_\infty}(u,u)$ is equivalent to the Dirichlet form
\begin{equation*}
    \form^{\mu(|\cdot|^2)} (u,u)
    :=\int_{\rn} \mu(|\xi|^2)|\widehat{u}(\xi) |^2\, d\xi,
\end{equation*}
where
\begin{equation*}
    \mu(t):=\int_0^t \int_r^\infty \frac{f(s)}{s^2}\, ds\, dr
\end{equation*}
is also a complete Bernstein function. It is easy to see that the Dirichlet form $\form^{\psi_\infty}(u,u)$ satisfies \eqref{n222} with $\delta=0$. Hence we get
\begin{lemma}
    Let $g$ and $\mu$ be as above and let $X$ be the L\'evy process with symbol $\psi_\infty(\xi)$ (or $\mu(|\xi|^2)$). Then $X$ has a transition density $p_t(x)$ and  there exist suitable constants $c, \gamma >0$ such that
\begin{equation}
    p_t(x)\leq c \left[ g^{-1}\left(\frac{1}{\gamma t}\right)\right]^{n/2}, \quad \text{for all $0<t<\infty$, $x\in\rn$}.\label{pf2}
\end{equation}
\end{lemma}

\section{An off-diagonal estimate for the transition density of L\'evy processes}\label{sec2}
In this section we prove an off-diagonal upper estimate for the transition density of a class of L\'evy processes. Let $X_t$ be a L\'evy process with symbol $\psi(\xi)$ and convolution semigroup $\mu_t(dy)$, $t\geq 0$. It is known that $\mu_t(dy)$ has a density $p_t(dy)$ with respect to Lebesgue measure if, and only if, $T_t f := f\star \mu_t$ is continuous for all Borel measurable functions $f$, cf.\ \cite[Lemmas 4.8.19, 4.8.20]{J01}. Apart from this criterion there are no good necessary and sufficient criteria for the existence of a transition density $p_t(x)$. Assume that $p_t(x)$ exists; since $\widehat p_t(\xi) = e^{-t\psi(\xi)}$, we have necessarily
\begin{equation}\label{den1}
    p_t(x)= \int_{\rn} e^{i\xi\cdot x -t\psi(\xi)}\, d\xi, \quad x\in\rn,\; t>0.
\end{equation}
This shows that the growth of $\psi$ as $|\xi|\to\infty$, e.g.\  $\psi(\xi)\geq |\xi|^\kappa$ for some $0<\kappa<2$, guarantees that \eqref{den1} converges absolutely. More generally, one has sufficient conditions due to Hartman-Wintner \cite{HW} and Kallenberg \cite{Ka}, see Bodnarchuk and Kulik \cite{BKu} for the $n$-dimensional situation.

Throughout we will assume that the L\'evy process $X_t$ has a transition density $p_t(x)$ for all $t>0$ and that the  L\'evy measure $\nu(dy)$ satisfies
\begin{align}\tag{\bfseries A1}\label{A1}
    \int_{|y|\geq 1}  e^{\alpha \cdot  y}\,\nu(dy)<\infty \quad
    \text{for all $\alpha\in\rn$}.
\end{align}
Note that \eqref{A1} is equivalent to saying that the L\'evy process has exponential moments, cf.\ Sato \cite[Theorem 25.3]{Sa}. Under \eqref{A1} the function
\begin{equation*}
    w(z)
    :=\frac{1}{2}\int_{\rn} \big(e^{z\cdot y}-1\big) \big(e^{-z\cdot y}-1\big)\, \nu(dy)
    =\int_{\rn}\big(1-\cosh(z\cdot y)\big)\,\nu(dy),
\end{equation*}
exists for any  $z\in \comp^n$, $z_j=\xi_j+i\eta_j$, $j=1,\ldots ,n$. In what follows, we write $\xi=(\xi_1,\ldots ,\xi_n)$, $\eta=(\eta_1,\ldots,\eta_n)$. Observe, that for $\xi=0$, i.e.\ $z=i\eta$, we have
\begin{equation}\begin{aligned}\label{Gam2}
    w(i\eta)
    &=\frac{1}{2}\int_{\rn}\big(e^{i\eta\cdot y}-1\big)\big(e^{-i\eta\cdot y}-1\big)\, \nu(dy)\\
    &=\int_{\rn}(1-\cos(y\cdot \eta))\,\nu(dy)
    =\psi(\eta).
\end{aligned}\end{equation}
This means that we have $w(\xi)=\Gamma(e^{ \xi\cdot }, e^{- \xi\cdot})$ for $\xi\in \rn$; here $\Gamma(u,u)$ is the carr\'e du champ operator associated with $\psi(\xi)$ by
\eqref{Gam1}.

Let us briefly recall  Carlen, Kusuoka and Stroock's upper bound for the transition density $p_t(x,y)$ of a general symmetric Markov process; for details we refer to \cite[Theorem 3.25]{CKS}.

\bigskip\noindent
{\itshape Let $(X_t)_{t\geq 0}$ be a symmetric Markov process given by a  regular Dirichlet form $(\form,D(\form))$ which admits a carr\'e du champ operator $\Gamma(\cdot,\cdot)$. Assume that there exists some $\phi\in D(\form)\cap L_\infty(\rn)\cap C_b(\rn)$ such that
\begin{gather*}
    \gamma(\phi)
    :=\sqrt{\left\| e^{-2\phi}\,\Gamma(e^\phi,e^\phi)\right\|_{L_\infty}}
    \vee \sqrt{\left\|e^{2\phi}\,\Gamma(e^{-\phi},e^{-\phi})\right\|_{L_\infty}}
    <\infty.
\end{gather*}
If the  Dirichlet form $(\form,D(\form))$ satisfies a Nash inequality with $f(x)=x^\alpha$,
i.e.\ if
$$
    \|u\|_{L_2}^{2\alpha / n}
    \leq A \, \big(\form(u,u)+\delta \|u\|_{L_2}^2 \big) \, \|u\|_{L_1}^{2\alpha /n},
$$
then the $(X_t)_{t\geq 0}$ has a transition density $p_t(x,y)$ and
\begin{equation}\label{off-d}
    p_t(x,y)\leq c_1\, t^{-n/\alpha}\, e^{c_2 t\gamma^2(\phi) -|\phi(x)-\phi(y)|+ c_3t}
\end{equation}
for all $t>0$ and almost all $x,y\in\rn$.}

\bigskip
Even in simple situations it is a non-trivial task to find concrete functions $\phi$ which also lead to reasonable estimates in \eqref{off-d}. Therefore we aim for a different way to derive a concrete off-diagonal upper bound for the transition density $p_t(x)$ of a L\'evy process. Let
\begin{equation*}
    Q_t(\xi,x):=i\xi\cdot x-t w(i\xi)=i\xi\cdot x-t\psi(\xi).
\end{equation*}
Observe that $Q_t(i\xi,x)$ resembles the exponent of the upper bound \eqref{off-d}. It is therefore a natural question whether it is possible to apply complex analysis techniques to get off-diagonal upper bounds similar to those in \eqref{off-d}.

Let us look closely at the properties of $w$. For  $\xi\in\rn$ we have:
\begin{enumerate}
\item[i)]
    $w|_{\rn}$ is even and $\nabla w(\xi) |_{\xi=0} =0$;
\item[ii)]
    $w(z)$ is an analytic function in each variable $z_j\in \comp$, $j=1,\ldots,n$; by Hartogs' theorem it is analytic in $\comp^n$.
\item[iii)]
    $w|_\rn$ is concave, i.e.\ for all $\xi,\xi'\in\rn$ and $t\in (0,1)$ one has $w(t \xi+(1-t)\xi')\geq t\, w(\xi)+(1-t)w(\xi')$. This follows directly from H\"{o}lder's inequality.
\item[iv)]
    $w|_\rn$ attains its maximum at $\xi=0$. This follows immediately from i) and iii).
\end{enumerate}
Consider the function
\begin{equation}\label{vt}
    v_t(\xi,x):=-\xi\cdot x-tw(\xi), \quad \xi\in\rn.
\end{equation}
Due to iii) it attains its minimum in $\xi$ at some point $\xi_0= \arg \min_{\xi} v_t(\xi,x)$. Since $v_t(0,x)=0$, we have $v_t(\xi_0,x)\leq 0$ for all $x\in\rn$, and $v_t(\xi_0,0)=0$. The following function plays a key role in the off-diagonal upper estimate. Write
\begin{equation}\label{D}
    D_t^2(x):=-v_t(\xi_0,x)
    \quad\text{where}\quad
    \xi_0 = \xi_0(t,x) = \arg\min_\xi v_t(\xi,x).
\end{equation}

\begin{theorem}\label{pr}
    Let $(X_t)_{t\geq 0}$ be a L\'evy process with symbol $\psi(\xi)$. Assume that $X_t$ has for all $t>0$ a transition density and assume that $\psi$ satisfies the Assumption \eqref{A1}.  Then
\begin{equation}\label{up1}
    p_t (x)\leq e^{-D_t^2(x)} p_t(0)
    \quad\text{for all}\quad x\in\rn,\; t>0,
\end{equation}
    where $D_t^2(x)$ is given by \eqref{D}.
\end{theorem}

For the proof of Theorem \ref{pr} we need a few properties of the function $D_t^2$.
\begin{lemma}
\begin{enumerate}
\item[\upshape i)]
    For all $t>0$, the function $x\mapsto D_t^2(x), x\in\rn$ is even and increasing as  $|x|\to\infty$;
\item[\upshape ii)]
    $D_t(x)=0$ if and only if $x=0$.
\item[\upshape iii)]
    $\displaystyle D_t^2(x)\leq \frac{|x|^2}{4ct}$ for all $x\in\rn$ and $t>0$.
\end{enumerate}
\end{lemma}
\begin{proof}
i) Since $x\mapsto\Gamma(e^{\xi\cdot x}, e^{-\xi\cdot x})$ is even,
$$
    \min_\xi v_t(\xi,x)=\min_\xi v_t(-\xi,-x)=\min_{-\xi} v_t(\xi,-x),
$$
and therefore $\xi_0(t,x)=-\xi_0(t,-x)$. This can be used to show
\begin{align*}
    v_t(\xi_0(t,x),x)
    &=-\xi_0(t,x)\cdot x-tw(\xi_0(t,x))\\
    &=\xi_0(t,-x)\cdot x -t w(-\xi_0(t,x)) \\
    &=-\xi_0(t,-x)\cdot (-x) -tw(\xi_0(t,-x))\\
    &=v_t(\xi_0(t,-x),-x),
\end{align*}
which implies that $x\mapsto D_t^2(x)$ is even. That $D_t^2(x)$ increases as $|x|\to\infty$ follows directly from the definition.

\medskip
ii) Since $v_t(\xi_0,x)$ is non-positive and since $v_t(\xi_0,0)=0$ it is obvious that $D_t(x)=0$ if, and only if, $x=0$.

\medskip
iii) By Taylor's theorem there exists a constant $c>0$ such that
$$
    \int_{\rn} \big(\cosh(\xi\cdot y)-1\big) \nu(dy)\geq c |\xi|^2
    \quad\text{for all $\xi\in\rn$}.
$$
Then for all $x\in\rn$
$$
    v_t(\xi,x)=-\xi\cdot x- tw(\xi) \geq -\xi\cdot x+ct|\xi|^2;
$$
if we minimize this expression, we get $-D_t^2(x) = v_t(\xi_0,x)\geq-|x|^2/(4ct)$.
\end{proof}

We are now ready for the
\begin{proof}[Proof of Theorem \ref{pr}]
    In order to estimate  $\int_{\rn}e^{Q_t(\xi,x)}\, d\xi$ we apply the Cauchy-Poincar\'e theorem. Since $w$ can be extended analytically to $\comp^n$, the function $Q_t(z,x)$ is analytic in $z\in\comp^n$. Without loss of generality we may assume that $\xi_0>0$, in the case $\xi_0<0$ the arguments are similar. Consider the domain
\begin{equation*}
    G
    :=\left\{z\in \comp^n \::\:   \Im z =t\xi_0,\; 0\leq t\leq 1,\; \Re z\in \prod_{j=1}^n [-M_j,M_j],\; M_j>0,\;1\leq j\leq n \right\}.
\end{equation*}
This is an $n+1$-dimensional cube with base $\left\{z\in \comp^n\::\: \Re z\in \prod_{j=1}^n [-M_j,M_j], \; \Im z=0\right\}$ and lid $\left\{z\in \comp^n\::\: \Re z\in \prod_{j=1}^n [-M_j,M_j],\; \Im z=\xi_0\right\}$. Since the number of sides of $G$ is even, we can fix some orientation on $\partial G$ such that base and lid have opposite orientation. 
By the Cauchy-Poincar\'e theorem
\begin{equation}\label{inc}
    \int_{\partial G} e^{Q_t(z,x)}\, dz_1 \wedge dz_2\wedge \cdots \wedge dz_n=0.
\end{equation}
Consider the integrals over the sides (except the base and the lid)
\begin{equation}\label{c2}
    \int_0^1 e^{Q_t(M\pm is\xi_0, x)}ds,
    \quad \text{where}\quad M=(\pm M_1,\ldots,\pm M_n)
\end{equation}
and recall that $Q_t(\xi,x)=i\xi\cdot x - tw(i\xi)$. After some rearrangements we get
\begin{align*}
    &\Re Q_t(M+i\eta,x)\\
    &= \Re \left[ i(M+i\eta)\cdot x-\frac{t}{2} \int_{\rn}\big(e^{i(M+i\eta)\cdot y}-1\big) \big(e^{-i(M+i\eta)\cdot y}-1\big)\,\nu(dy)\right]\\
    &=-\eta\cdot x -t \int_{\rn} \big(1-\cosh(\eta\cdot y)\big)\,\nu(dy)-t \int_{\rn} \cosh(\eta\cdot y)(1-\cos (M\cdot y))\,\nu(dy)\\
    &\leq -t\psi(M)-\eta\cdot x -t \psi(i\eta).
\end{align*}
Therefore,
\begin{align*}
    \left| e^{Q_t(M\pm is \xi_0, x)}\right| \leq e^{-t\psi(M)-\eta\cdot x -t \psi(is \xi_0)},
\end{align*}
which means that the  integrands in \eqref{c2} tend, uniformly in $s$, to $0$ as $|M|\to \infty$. Therefore, \eqref{inc} becomes, as $|M|\to\infty$,
\begin{equation*}
    \int_{\rn} e^{Q_t(\xi,x)}\, d\xi
    =\int_{\rn} e^{Q_t(\xi+ i\xi_0,x)}\, d\xi.
\end{equation*}
Since
\begin{align*}
    \Re Q_t(\xi+ i\xi_0,x)
    &=v_t(\xi_0,x)-t\int_{\rn} \cosh(\xi_0\cdot y) (1-\cos (\xi\cdot y))\, \nu(dy)\\
    &\leq v_t(\xi_0,x) -t\psi(\xi),
\end{align*}
we finally get
\begin{gather*}
    \int_{\rn} e^{Q_t(\xi,x)}\, d\xi
    \leq e^{v_t(\xi_0,x)} \int_{\rn}e^{-t\psi(\xi)}\, d\xi
    = e^{-D_t^2(x)}p_t(0).
\qedhere
\end{gather*}
\end{proof}

We can combine Theorem \ref{pr} with the on-diagonal estimates from Section \ref{sec1}, e.g.\ with Lemma \ref{nash}. Note that all processes satisfying the assumptions of Lemma \ref{nash} automatically have transition densities.
\begin{corollary}\label{cor8}
    Let $(X_t)_{t\geq 0}$ be a L\'evy process with symbol $\psi_1$ as in \eqref{psi1} where $f$ is a Bernstein function without linear term and such that $f(0)=0$. Then  there exist suitable constants $c, \gamma >0$ such that
\begin{equation}\label{up-psi}
    p_t (x)\leq c e^{-D_t^2(x)} \left[ f^{-1}\left(\frac{1}{\gamma t}\right)\right]^{n/2},
    \quad\text{for all\ \ } 0<t\leq 1,\; x\in\rn.
\end{equation}
\end{corollary}
Since the representing measure $\nu(dy)$ has compact support, \eqref{A1} is satisfied and the proof follows by Lemma~\ref{nash} and Theorem~\ref{pr}. Note that we do not require the growth condition imposed on $f$ in Lemma \ref{nash}, since this was only used to show the equivalence of the Dirichlet forms $\form^{\psi_1}$ and $\form^{f(|\cdot|^2)}$.

\begin{example}
    Let us indicate a few generic examples for our estimates. Since we require exponential moments, our L\'evy measures $\nu$ must satisfy $\int_{|y|\geq 1} e^{\xi\cdot y}\,\nu(dy) < \infty$ for all $\xi\in\real^n$. Typically, this can be achieved if $\supp\nu$ is bounded, see Examples i) and ii) below, or if $\nu$ has an exponentially fast decaying density, as in Example iii). For simplicity we restrict ourselves to the one-dimensional setting $n=1$ and estimates for $p_t(x)$ with $x/t\to\infty$, e.g.\ where $x\gg 1$ and $t>0$ is fixed.
    \begin{enumerate}
    \item[i)]
        Assume that $\supp\nu\subset[-1,1]$ and that $\nu(dy)=g(y)\,dy$ for some density function $g(y)$. This covers the situation of \eqref{psi1} and Corollary \ref{cor8}.

        For our heat kernel estimate we need to control the behaviour of the function $v_t(\xi_0,x)$ from \eqref{vt} at the point $\xi_0=\arg\min_\xi v_t(\xi,x)$. For symmetry reasons we only have to consider the case where $\xi >0$ and $x>0$. Clearly,
        $$
            v_t(\xi,x)
            \leq -\xi x + c_0 t \xi^2 e^\xi \int_{-1}^1 y^2\,g(y)\,\nu(dy)
            \leq -\xi x + c_1 t  e^{\xi(1+\varepsilon)}
        $$
        for some for some $\varepsilon >0$ and suitable constants $c_0$ and  $c_1$. We will now minimize the expression on the right-hand side. From
        $$
            \frac{\partial}{\partial \xi}\left(-\xi x + c_1 t  e^{\xi(1+\varepsilon)}\right)
            = -x  + t c_1 (1+\varepsilon)\, e^{\xi(1+\varepsilon)}
            \stackrel{!}{=} 0
        $$
        we find the critical point for the minimum of the right-hand side
        $$
            \xi=\frac{1}{1+\varepsilon}\ln \left(\frac{x}{t c_1 (1+\varepsilon)} \right),
        $$
        and so
        $$
            v_t(\xi_0,x)
            \leq \min_\xi \Big( -\xi x + c_1 t  e^{\xi(1+\varepsilon)}\Big)
            = -\frac{x}{1+\varepsilon}\ln \left(\frac{x}{t c_1 (1+\varepsilon)} \right) +\frac{x}{1+\varepsilon}.
        $$

        If, as in Corollary \ref{cor8}, $g(y)=f\big(|y|^{-2}\big)/|y|^n$ the upper bound becomes for $x\neq 0$
        \begin{align*}
            p_t(x)
            &\leq c\, e^{-D_t^2(x)} \left[ f^{-1}\left(\frac{1}{\gamma t}\right)\right]^{1/2}\\
            &\leq e^{-\frac{|x|}{1+\varepsilon}\ln \left(\frac{|x|}{t c_1 (1+\varepsilon)} \right) +\frac{|x|}{1+\varepsilon}} \left[ f^{-1}\left(\frac{1}{\gamma t}\right)\right]^{1/2}
        \end{align*}
        for suitable constants $\gamma, \varepsilon > 0$.

    \item[ii)]
        Assume that $\supp\nu\subset[-1,1]$ and that $\nu$ is discrete. Let us consider the case where $\nu(dy)=\sum_{n=0}^\infty 2^{\alpha n} \left(\delta_{2^{-n}}+\delta_{-2^{-n}}\right) $ and $0<\alpha<2$. Since $\nu_0 = \sum_{n=-\infty}^\infty 2^{\alpha n} (\delta_{2^{-n}}+\delta_{-2^{-n}})$ corresponds to the so-called \emph{$\alpha$-semi-stable  process}, see \cite[Example~13.3]{Sa}, we can use \cite[Proposition 24.20]{Sa} to get $\psi_0(\xi)\geq c\, |\xi|^\alpha$ for some $c>0$ and all $|\xi|\geq 1$. Since the characteristic exponents $\psi$ and $\psi_0$ corresponding to $\nu$ and $\nu_0$ satisfy $\psi(\xi) \asymp \psi_0(\xi)$ for large values of $|\xi|$, we also have $\psi(\xi)\geq c'\,|\xi|^\alpha$, $|\xi|\geq 1$. This means, in particular, that a transition density $p_t(x)$ exists.

        Note that all our calculations for the upper estimate for $v_t(\xi,x)$ from the first example remain valid and, what is more, up to the constant $c_1 = c_1(\nu)$ they depend only on the (size of the) support of the L\'evy measure, but not on the particular form of $\nu$. This means that we can also in this case estimate the transition density by
        $$
            p_t(x)
            \leq c_2\,{t^{-\frac{1}{\alpha}}}\, e^{-\frac{|x|}{1+\varepsilon}\ln \left(\frac{|x|}{t c_1 (1+\varepsilon)} \right) + \frac{|x|}{1+\varepsilon}}.
        $$

    \item[iii)]
        Assume that $\nu(dy)= \nu_0(dy) + \I_{|y|\geq 1}\,e^{-|y|^\beta}\,dy$ where $\beta>1$ and where $\nu_0$ is as in Example i) or ii).

        In this case the behaviour of $v_t(\xi,x)$ is determined by the tail of the measure. As before, we first estimate $v_t(\xi,x)$ from above for $x>0$ and $\xi>0$:
        \begin{equation*}
            v_t(\xi,x)
            \leq -\xi x + c_1 t \,e^{\xi (1+\varepsilon)} +c_2 t \int_1^\infty e^{\xi y-|y|^\beta}dy .
        \end{equation*}
        In order to find the asymptotics of $I_1(\xi):=\int_1^\infty e^{\xi y-|y|^\beta}\,dy$ we use the Laplace method, see \cite[\S 18, p.\ 58]{Cop}. It is known that for sufficiently smooth functions $h$ the integral $I_h(\xi):=\int_a^b e^{h(\xi,y)}\,dy$ where $a,b\in [-\infty,+\infty]$ satisfies
        $$
            I_h(\xi)\sim \sqrt{\frac{\pi}{2 |h''(\xi,y_0)|}} \, e^{h(\xi,y_0)}, \quad\text{as}\quad \xi\to\infty.
        $$
        In the expression above $y_0$ is the (unique) point where $h$ reaches its maximum. If we use $h(\xi,y)=\xi y -y^\beta$ and $y_0=\left(\frac{\xi}{\beta}\right)^{\frac{1}{\beta-1}}$, we get
        \begin{align*}
            h(\xi,y_0)
            &=-c_{\beta,1}\,\xi^{\frac{\beta}{\beta-1}}
            \quad\text{where}\quad
            c_{\beta,1}
            = (\beta-1)\beta^{\frac \beta{1-\beta}}\\
            |h''(\xi,y_0)|
            &= \beta(\beta-1) \beta^{\frac{1}{1-\beta}}\, \xi^{\frac{\beta-2}{\beta-1}},\\
            I_1(\xi)
            &\asymp c_{\beta,2}\,\xi^{\frac{2-\beta}{2(\beta-1)}}\, e^{c_{\beta,1}\xi^{\frac{\beta}{\beta-1}}}, \quad \xi\to\infty.
        \end{align*}
        Observe that for $\xi\to\infty$
        $$
            v_t(\xi,x)
            \leq f(\xi,x,t )= -\xi x + t\,c_{\beta,2}\,\xi^{\frac{2-\beta}{2(\beta-1)}} \, e^{c_{\beta,1}\xi^{\frac{\beta}{\beta-1}}}.
        $$
        The point where $f(\cdot,x,t)$ becomes extremal satisfies the equation
        \begin{equation*}
        \begin{aligned}
        &0\stackrel{!}{=} f'_\xi (\xi,x,t)\\
        &=-x+ t\,c_{\beta,2}\,\frac{2-\beta}{2(\beta-1)}\,\xi^{\frac{2-\beta}{2(\beta-1)}-1} e^{c_{\beta,1}\,\xi^{\frac{\beta}{\beta-1}}}
        +t \, c_{\beta,2} c_{\beta,1} \,\frac{\beta}{\beta-1}\, \xi^{\frac{\beta}{\beta-1} + \frac{2-\beta}{2(\beta-1)} -1}
        \,e^{c_{\beta,1}\,\xi^{\frac{\beta}{\beta-1}}}.
        \end{aligned}
        \end{equation*}
        Rather than solving this equation for $\xi$ explicitly, we determine the asymptotic behaviour of the solution as $x/t\to\infty$. Note that $f'_\xi (\xi,x,t)=0$ if, and only if,
        $$
            c_{\beta,2}\, \frac{2-\beta}{2(\beta-1)}\,\xi^{\frac{2-\beta}{2(\beta-1)}-1} \, e^{c_{\beta,1}\xi^{\frac{\beta}{\beta-1}}}
            + c_{\beta,2} c_{\beta,1}\, \frac{\beta}{\beta-1}\, \xi^{\frac{\beta}{\beta-1} + \frac{2-\beta}{2(\beta-1)} -1}\, e^{c_{\beta,1}\xi^{\frac{\beta}{\beta-1}}}
            =
            \frac{x}{t}.
        $$
        Taking logarithms on both sides we arrive at
        %
        $$
            \xi=
            \left(\frac{1}{c_{\beta,1}}\log \big(\frac{x}{t}\big)\right)^{\frac{\beta-1}{\beta}} + o\left(\log \frac{x}{t}\right)
            \quad\text{as}\quad \frac{x}{t}\to \infty,
        $$
        hence,
        $$
            v_t(\xi_0,x)
            \leq -(1-\varepsilon) x\, \left(\frac{1}{c_{\beta,1}}\,\log \left(\frac{x}{t}\right)\right)^{\frac{\beta-1}{\beta}}.
        $$
        as well as
        $$
            p_t(x)
            \leq p_t(0) e^{ -(1-\varepsilon) x \left(\frac{1}{c_{\beta,1}}\,\log \left(\frac{x}{t}\right)\right)^{\frac{\beta-1}{\beta}}}.
        $$
        With considerably more effort it is possible to obtain the exact asymptotics of $v_t(\xi_0,x)$ and $p_t(x)$, see \cite[Proposition~6.1]{KK} by Kulik and one of the present authors.
    \end{enumerate}
\end{example}

\section{An application to large deviations}\label{sec3}
In this section we will show an application of Theorem~\ref{pr} to the theory of large deviations. We show that the transition density \eqref{den1} satisfies the \emph{large deviation principle} (LDP) with the rate function $D_t^2(x)$.

Let us briefly recall the LDP. Let $(X_t)_{t\geq 0}$ be a L\'evy process associated with transition function $\mu_t(dx)$. Moreover, we assume that $X_t$ has exponential moments, i.e.\
\begin{gather}\tag{\textbf{A2}}\label{A2}
    \int_{\rn}  e^{y\cdot \lambda}\mu_t(dy)<\infty
    \quad\text{for all $\lambda\in\rn$ and $t>0$}.
\end{gather}
By \cite[Theorem 23.5]{Sa}, this is equivalent to our assumption \eqref{A1}. Therefore, we
can extend $\psi$ analytically from $\rn$ to $\comp^n$.

Let
\begin{equation*}
    \Lambda^*_\mu(x,t):=\sup_{\xi}\{  \xi\cdot x -\Lambda_\mu (\xi,t)\},
\end{equation*}
where
\begin{equation*}
    \Lambda_\mu(\xi,t):=\log \int_{\rn} e^{\xi\cdot y} \mu_t(dy)=t\psi(i\xi).
\end{equation*}
By $\mu_t^{(\ell)}(dx)$ we denote  the probability measure related to  $Y_t^{(\ell)}:=\frac{1}{\ell} \sum_{j=1}^\ell X_{t}^{j}$, where $X_t^j$ are independent copies of $X_t$. It is known, see e.g.\ \cite[Chapter 3]{FK}, that under \eqref{A1} the sequence of measures $(\mu_t^{(\ell)}(dx))_{\ell\geq 1}$ is exponentially tight and, by Cramer's theorem, it satisfies the LDP with \emph{good rate function} $\Lambda^*_\mu(x,t)$, i.e.\ for all measurable subsets $B\subset\rn$ the inequalities
\begin{equation*}
    -\inf_{x\in B^\circ} \Lambda^*_\mu(x,t)
    \leq \varliminf_{\ell\to\infty} \frac{1}{\ell} \log \mu_t^{(\ell)}(B)
    \leq \varlimsup_{\ell\to\infty} \frac{1}{\ell} \log \mu_t^{(\ell)}(B)
    \leq - \inf_{x\in \overline{B}} \Lambda^*_\mu(x,t).
\end{equation*}
hold; $B^\circ$ and $\overline{B}$ are the open interior and the closure of $B$, respectively. Clearly,  $\Lambda^*_\mu(x,t)= -v_t(\xi_0,x)=D_t^2(x)$. By Theorem~\ref{pr} we have the analogue of the LDP for the transition density $p_t(x)$ from Theorem~\ref{pr}.

\begin{proposition}\label{prop9}
    Let $(X_t)_{t\geq 0}$ be a L\'evy process with symbol $\psi(\xi)$ satisfying
\begin{equation}\label{loginc}
    \lim_{|\xi|\to\infty} \frac{\psi(\xi)}{\log(1+|\xi|)} > C.
\end{equation}
    Assume that \eqref{A1}, or equivalently \eqref{A2}, holds. Then $X_t$ has for all $t>t_0:= n/C$ a transition density $p_t(x)$ and for all $t>t_0$
\begin{equation}
    \lim_{\ell\to \infty} \frac{\log p_{\ell t}(\ell x)}{\ell}=-D_t^2(x).
\end{equation}
\end{proposition}

\begin{proof}
Condition \eqref{loginc} is the Hartman-Wintner condition which ensures that $X_t$ has a (continuous) transition density for all $t>t_0 = n/C$, see \cite{HW}. Clearly, we may follow the arguments of the proof of Theorem \ref{pr} whenever $p_t(x)$ exists, i.e.\ for $t>t_0$.

As in the proof of Theorem~\ref{pr}, we can write the transition density $p_t(x)$ as
$$
    p_t(x)=e^{v_t(x,\xi_0)} \int_{\rn} e^{A_{t,x,\xi_0}(\xi)} \, d\xi,
$$
where $\xi_0$, $v_t(x,\xi_0)$ are as before, and $A_{t,x,\xi_0}:=Q_t(\xi+i\xi_0,x)-v_t(\xi_0,x)$. Note that
\begin{equation}\label{xi00}
    \left|\int_{\rn} e^{A_{t,x,\xi_0}(\xi)}\, d\xi\right|
    \leq\left|\int_{\rn} e^{A_{t,0,\xi_0}(\xi)}\, d\xi\right|
    \leq  p_t(0).
\end{equation}
By \eqref{loginc} we obtain for $k\geq t_0$
\begin{align*}
    \frac{\log p_k(0)}{k}
    \leq \log\left( \int_{\rn} e^{-kC'\log(1+|\xi|)}d\xi \right)^{1/k}
    &= \log \left(\int_{\rn} \frac{d\xi}{(1+|\xi|)^{kC'}}\right)^{1/k}\\
    &\asymp \frac 1k \log \frac{1}{kC'}\xrightarrow{k\to\infty} 0.
\end{align*}
Because of \eqref{xi00}
\begin{equation*}
    \lim_{\ell\to\infty} \frac{\log \int_{\rn} e^{A_{\ell t,x,\xi_0}(\xi)}\, d\xi }{\ell}=0,
\end{equation*}
and since $v_{\ell t}(\xi,\ell x)=-\ell\xi\cdot x-\ell t w(\xi)$, it is clear that $\xi_0=\arg \min_\xi v_t(\xi,x)$ does not depend on $\ell$. Hence,
\begin{equation*}
    -D_{\ell t}^2 (x)=v_{\ell t}(\xi_0,\ell x)=\ell v_t(\xi_0,x)=-\ell D_t^2(x).
\end{equation*}
Combining the last two formulae we get
\begin{gather*}
    \lim_{\ell\to \infty} \frac{\log p_{\ell t}(\ell x)}{\ell}
    = \lim_{\ell\to\infty}\frac{\log \int_{\rn} e^{A_{\ell t,x,\xi_0}(\xi)}\, d\xi}{\ell} + \lim_{\ell\to\infty} \frac{\log e^{-\ell D_t^2(x)}}{\ell}
 =-D_t^2(x).
\qedhere
\end{gather*}
\end{proof}

\section{Estimates for L\'evy-type processes}\label{sec4}

In this section we generalize the results obtained in  Section~\ref{sec2} to the case of pseudo-differential operators with continuous negative definite symbol. We show that under some conditions one can construct an upper bound in the form similar to \eqref{up1} for the transition density of a Markov process related to a pseudo-differential operator. Typical examples are Feller processes such that the test functions $C_0^\infty$ are in the domain $D(A)$ of their generator, cf.\ \cite{J01,JSS}.

Let $u\in C_0^\infty(\rn)$. Consider the operator
\begin{equation}\label{qD}
    q(x,D) u(x)=\int_{\rn} e^{i\xi\cdot x} q(x,\xi)\,\widehat{u}(\xi)\, d\xi,
\end{equation}
where $q:\rn\times \rn \to\real$ is locally bounded and for each $x\in\rn$ the function $q(x,\cdot)$ is continuous negative definite. This means that  $q(x,\xi)$ admits a L\'evy-Khintchine representation:
\begin{equation}\label{lhr}
    q(x,\xi)
    = \frac 12\,\xi\cdot Q(x)\xi + \int_{\rn\setminus\{0\}} \left(1-\cos(y\cdot\xi)\right)\, N(x,dy),\quad x\in\rn,
\end{equation}
where  $Q(x)=(Q_{jk}(x))\in\real^{n\times n}$ is a positive semi-definite matrix, and $N(x,dy)$ is a L\'evy kernel, i.e.\ for fixed $x\in\rn$ it is a Borel measure on $\rn\setminus \{0\}$, such that
$$
    \int_{\rn\setminus \{0\}} (|y|^2\wedge 1 )\,N(x,dy)<\infty.
$$
Such an operator $q(x,D)$  is called a pseudo-differential operator with real-valued continuous negative definite symbol. Using Fourier inversion the following integro-differential representation for $u\in C_0^\infty(\rn)$ is easily derived:
\begin{equation}\label{qD-integro}
    -q(x,D) u(x)= \frac 12\,\sum_{j,k=1}^n Q_{jk}(x)\partial_j\partial_k u(x) + \int_{\rn\setminus\{0\}}\left(u(x+y)-u(x)\right)N(x,dy).
\end{equation}

In the sequel we will need a few further assumptions.
\begin{gather}\tag{\bfseries B1}\label{B1}
    (-q(x,D),C_0^\infty(\rn))
    \text{\ extends to the generator $(A,D(A))$ of a Feller semigroup}.
\end{gather}

\begin{remark}
    For sufficient conditions when \eqref{B1} is satisfied we refer to \cite[Theorem~5.2]{J94}, \cite[Theorem 4.14]{Hhab}, \cite[Section 2.6]{J02} or the survey paper \cite{JSS} and the references given there.
\end{remark}

Denote by $(T_t)_{t\geq 0}$ the Feller semigroup and by $(X_t)_{t\geq 0}$ the Feller process generated by (the extension of) $-q(x,D)$. We set
\begin{equation} \label{lam}
    \lambda_t(x,\xi):=e^{-i\xi\cdot x}T_t (e^{i\xi \bullet})(x).
\end{equation}
Since $q(x,\xi)$ is real-valued, $\lambda_t(x,\xi)=\lambda_t(x,-\xi)$ for all $x,\xi\in\rn$ and $t>0$. Writing $p_t(x,dy)$ for the transition function of the process, it is easy to see that each $T_t$ is a pseudo-differential operator with symbol $\lambda_t(x,\xi)$:
\begin{equation}\label{Tlam}
    T_t u(x)
    =\int_{\rn} u(y)p_t(x,dy)
    =\int_{\rn} e^{i\xi\cdot x}\lambda_t(x,\xi)\,\widehat{u}(\xi)\,d\xi,
    \quad u\in C_0^\infty(\rn).
\end{equation}
From \eqref{Tlam} we see that if $\lambda_t(x,\cdot)\in L_1(\rn)$, the probability measures $p_t(x,dy)$ have densities $p_t(x,y)$ w.r.t.\ Lebesgue measure $dy$ and we see, cf.\ \cite[Theorem 3.2.1]{J01}, that
\begin{align*}
    \|T_t u\|_{L_\infty}
    \leq \|\lambda_t(x,\cdot)\|_{L_1} \cdot \|\widehat{u}\|_{L_\infty}
    \leq \|\lambda_t(x,\cdot)\|_{L_1}\cdot \|u\|_{L_1}.
\end{align*}
Thus,
\begin{equation}
    p_t(x,y)\leq  \|\lambda_t(x,\cdot)\|_{L_1},
    \quad\text{for all}\quad x,y\in\rn,\;t>0,
\end{equation}
and we can express $p_t(x,y)$ in terms of the symbol $\lambda_t(x,\xi)$:
\begin{equation}\label{plam}
    p_t(x,y)=\int_{\rn} e^{i \xi\cdot(x-y)} \lambda_t(x,\xi)\,d\xi.
\end{equation}

In addition to \eqref{B1} we will also need
\begin{gather}
    \lambda_t(x,\cdot)\in L_1(\rn)
    \quad\text{for all $x\in\rn$, $t>0$}.
    \tag{\bfseries B2}\label{B2}\\
    \sup_{z\in\rn} \int_{0<|y|\leq 1} |y|^2 N(z,dy) + \sup_{z\in\rn} \int_{|y|>1} e^{\zeta\cdot y} N(z,dy) < \infty
    \quad\text{for all $\zeta\in\rn$}.
    \tag{\bfseries B3}\label{B3}\\
    \left| \frac{\lambda_t(x,\eta+i\xi)}{\lambda_t(x,i\xi)}\right| \leq |\lambda_t(x,\eta)|
    \quad\text{for all $\xi,\eta\in\rn$, $t>0$}.
    \tag{\bfseries B4}\label{B4}
\end{gather}

Note that \eqref{B3} entails that the process $(X_t)_{t\geq 0}$ has exponential moments.
\begin{lemma}\label{lemma-exp}
    Let $-q(x,D)$ be a pseudo-differential operator satisfying \eqref{B1} and \eqref{B3}. Then the Feller process $(X_t)_{t\geq 0}$ generated by (the extension of) $-q(x,D)$ admits exponential moments:
    \begin{equation}\label{B33}
        \Ee^x e^{\zeta\cdot X_t} = \int_\rn e^{\zeta\cdot y}\,p_t(x,y)\,dy < \infty
    \end{equation}
    for all $\zeta\in\rn$ and $x\in\rn$.
\end{lemma}
\begin{proof}
    Write $(A,D(A))$ for the generator of the process and its semigroup $(T_t)_{t\geq 0}$. Since $(X_t)_{t\geq 0}$ is a Feller process we know that for all $u\in D(A)$ the process $M_t^u := u(X_t)-u(X_0)-\int_0^t Au(X_s)\,ds$ is a martingale. In particular,
    \begin{equation}\label{Tq}
        T_t u(x)-u(x)=-\int_0^t T_s q(x,D) u(x)\, ds,
    \end{equation}
    for all $u\in C_0^\infty(\rn)$.

    Let us show that \eqref{Tq} still holds for $u(x) := u_{z,\zeta}(x):=\cosh \big(\zeta\cdot(x-z)\big)$, $\zeta\in\rn$. For this pick a sequence of test functions $\chi_\ell\in C_0^\infty(\rn)$, $\ell\in\nat$, with $\I_{B(0,\ell)}\leq\chi_\ell\leq\I_{B(0,2\ell)}$; observe that $x\mapsto \chi_\ell(x-z) \cosh \big(\zeta\cdot(x-z)\big)$ is in $D(A)$. Therefore we can rewrite \eqref{Tq} in the following way
\begin{multline*}
    \Ee^x\Big[\chi_\ell(X_t-z)\cosh\big(\zeta\cdot(X_t-z)\big)\Big]\bigg|_{z=x}-1\\
    = \Ee^x \int_0^t \int\limits_{\rn\setminus\{0\}} \Big[\chi_\ell(X_s-x+y)\cosh\big(\zeta\cdot(X_s-x+y)\big) - \chi_\ell(X_s-x)\cosh\big(\zeta\cdot(X_s-x)\big)\Big]\\
    \hfill\times N(X_s,dy)\,ds
\end{multline*}
Since we are integrating with respect to Lebesgue measure $ds$ and since $s\mapsto X_s$ has almost surely at most countably many jumps, we may replace $X_s$ in the above formula by its left limit $X_{s-}$. Moreover, set
$$
    \tau_k := \inf\left\{s\geq 0\::\: |X_s - x|\geq k\right\}.
$$
Using optional stopping for the martingale $M_t^u$ and writing $X^{\tau_k}_s := X_{\tau_k\wedge s}$ for the stopped process, we get from the identity above
\begin{align}
    &\Ee^x\Big[\chi_\ell(X_t^{\tau_k}-x)\cosh\big(\zeta\cdot(X_t^{\tau_k}-x)\big)\Big]-1\notag\\
    &= \left|\Ee^x \int_0^{t\wedge\tau_k} \int_{\rn\setminus\{0\}} \left(\cdots \right)N(X_{s-},dy)\,ds\right|\notag\\
    &\leq \Ee^x \int_0^{t\wedge\tau_k} \int_{\rn\setminus\{0\}} \left|\left(\cdots \right)\right|N(X_{s-},dy)\,ds\notag\\
    &\leq \Ee^x \int_0^{t\wedge\tau_k} \!\!\! \int\limits_{0<|y|\leq 1} \left|\left(\cdots \right)\right| N(X_{s-},dy)\,ds
    + \Ee^x \int_0^{t\wedge\tau_k} \!\!\!\int\limits_{|y|>1} \left|\left(\cdots \right)\right|N(X_{s-},dy)\,ds.
    \label{e-terms}
\end{align}
Let us estimate the last two integrals separately. For $\ell\geq k+1$, $|y|\leq 1$ and $s\leq\tau_k\wedge t$ we have $\chi_\ell(X_{s-}-x-y)=\chi_\ell(X_{s-}-x)=1$. Thus,
\begin{align*}
    &\Ee^x \int_0^{t\wedge\tau_k} \int_{0<|y|\leq 1} \left|\left(\cdots \right)\right|N(X_{s-},dy)\,ds\\
    &= \Ee^x \int_0^{t\wedge\tau_k} \int_{0<|y|\leq 1} \left|\cosh\big(\zeta\cdot(X_{s-}-x+y)\big) - \cosh\big(\zeta\cdot(X_{s-}-x)\big)\right|N(X_{s-},dy)\,ds\\
    &\leq \Ee^x \int_0^{t\wedge\tau_k} \int_{0<|y|\leq 1} \cosh\big(\zeta\cdot(X_{s-}-x)\big)\cdot \big(\cosh(\zeta\cdot y) - 1\big)\,N(X_{s-},dy)\,ds\\
    &\leq \frac 12|\,\zeta|^2 \,e^{|\zeta|} \,\Ee^x \int_0^{t\wedge\tau_k} \int_{0<|y|\leq 1} \cosh\big(\zeta\cdot(X_{s-}-x)\big) \, |y|^2\,N(X_{s-},dy)\,ds\\
    &\leq \frac 12\,|\zeta|^2 \, e^{|\zeta|} \,t\,\sup_{s<t} \Ee^x \big(\cosh\big(\zeta\cdot(X^{\tau_k}_{s-}-x)\big)\big) \, \sup_{z\in\rn} \int_{0<|y|\leq 1} |y|^2\,N(z,dy)\,ds
\end{align*}
where we used the elementary inequalities $|\cosh(a+b)-\cosh(a)|\leq \cosh(a) (\cosh(b)-1)$ and $\cosh(b)-1 \leq \frac 12\,b^2\,e^{|b|}$.

A similar calculation using the fact that $|\chi_\ell|\leq 1$ and $\cosh(a+b)\leq e^{|a|}\cosh(b)$ yields for the second integral term in \eqref{e-terms}
\begin{align*}
    &\Ee^x \int_0^{t\wedge\tau_k} \int_{|y|>1} \left|\left(\cdots \right)\right| N(X_{s-},dy)\,ds\\
    &\leq \Ee^x \int_0^{t\wedge\tau_k} \int_{|y|> 1} \Big(\left|\cosh\big(\zeta\cdot(X_{s-}-x+y)\big)\right| + \left|\cosh\big(\zeta\cdot(X_{s-}-x)\big)\right| \Big) N(X_{s-},dy)\,ds\\
    &\leq t\,\sup_{s<t} \Ee^x \big(\cosh\big(\zeta\cdot(X^{\tau_k}_{s-}-x)\big)\big)\,\sup_{z\in\rn}\int_{|y|>1} \left(e^{|\zeta\cdot y|}+1\right)N(z,dy).
\end{align*}
Because of \eqref{B3} we find a constant $C=C_\zeta$ not depending on $k$ or $x$ such that for all $t>0$
\begin{align*}
    \Ee^x \left[\cosh\big(\zeta\cdot(X_t^{\tau_k}-x)\big)\right]
    &= \sup_{\ell\in\nat} \Ee^x \left[\chi_\ell(X_t^{\tau_k}-x) \cosh\big(\zeta\cdot(X_t^{\tau_k}-x)\big)\right]\\
    &\leq 1 + t\,C\, \sup_{s\leq t} \Ee^x\left[\cosh\big(\zeta\cdot(X_{\tau_k\wedge s-}-x)\big)\right]\\
    &\leq 1 + t\,C\, \sup_{s\leq t} \left(\gamma_k \wedge\Ee^x \left[ \cosh\big(\zeta\cdot(X_{s}^{\tau_k}-x)\big)\right]\right).
\end{align*}
where $\gamma_k > \sup_{|y|\leq k}\cosh\big(\zeta\cdot(y-x)\big)$. Since $t>0$ was arbitrary and since the right-hand side depends monotonically on $t$, the above calculation also gives
$$
    \sup_{s\leq t}\Ee^x \cosh\big(\zeta\cdot(X_s^{\tau_k}-x)\big)
    \leq 1 + t\,C\, \sup_{s\leq t} \left(\gamma_k\wedge\Ee^x \cosh\big(\zeta\cdot(X_{s}^{\tau_k}-x)\big)\right).
$$
Estimating the left-hand side trivially from below, and choosing $t<t_0 < 1/(2C)$, we find
$$
    \sup_{s\leq t}\Big(\gamma_k\wedge \Ee^x \cosh\big(\zeta\cdot(X_{s}^{\tau_k}-x)\big)\Big)
    -\frac 12 \,\sup_{s\leq t}\Big(\gamma_k\wedge \Ee^x \cosh\big(\zeta\cdot(X_{s}^{\tau_k}-x)\big)\Big)
    \leq 1.
$$
Note that \eqref{B3} entails that the generator $-q(x,D)$ has bounded `coefficients', i.e.\ the life-time of $X_t$ is a.s.\ infinite and $\lim_{k\to\infty} \tau_k = \infty$, see \cite{sch-positivity}. By Fatou's Lemma we get
$$
    \sup_{s\leq t}\Ee^x \cosh\big(\zeta\cdot(X_{s}-x)\big)
    \leq 2\quad\text{for all}\quad x\in\rn,\; t\in (0,t_0].
$$
Note that $t_0=t_0(\zeta)$. Now it is a simple exercise using the Markov property to show that
\begin{gather*}
    \Ee^x\cosh\big(\zeta\cdot(X_t-x)\big) < \infty\quad\text{for all}\quad t>0.
\qedhere
\end{gather*}
\end{proof}

Here are some examples of Markov processes for which the conditions \eqref{B1}--\eqref{B4} are satisfied.
\begin{example}\label{je1}
    Let
    $$
        \mathcal{L}
        :=
        \sum_{j,k=1}^n \frac{\partial}{\partial x_j}\left(a_{jk}(x)\frac{\partial}{\partial x_k}\right),
    $$
    where $(a_{jk}(x))_{j,k= 1}^n$ is a symmetric  positive definite matrix with bounded measurable coefficients. It is known that there exists the transition density $p_t^{(0)}(x,y)$ of the diffusion semigroup  associated with $\mathcal{L}$, which  satisfies  Aronson's estimates, see \cite{Ar} and \cite{Da}, and which is jointly H\"older continuous  in $x$ and $y$. Let $\psi^{(0)}$ be a continuous negative definite function satisfying the conditions of Theorem~\ref{pr}. Then the operator $\mathcal{L}+\psi^{(0)}(D)$ satisfies \eqref{B1}--\eqref{B4}. Indeed, in this case the symbol $\lambda_t(x,\xi)$ associated with $\mathcal{L}+\psi^{(0)}(D)$ is the product of two symbols, associated with $\mathcal{L}$ and with $\psi^{(0)}(D)$, both satisfying \eqref{B1}--\eqref{B4}.
\end{example}

\begin{example}
    Consider a Markov process in  $\real^m_+\times \rn$, such that for every $t>0$ the characteristic function $\lambda_t(x,\xi)$ has exponential affine dependence on $x$. That is, for every $(t,\xi)\in \real_+\times i\real^{n+m}$ there exist $\Phi(t,\xi)\in \comp$, $\Psi(t,\xi)=(\Psi^Y(t,\xi),\Psi^Z(t,\xi)) \in \comp^m\times \comp^n$, such that for all $x\in\real^m_+\times \rn$
    \begin{equation}
        \lambda_t(x,\xi)= e^{\Phi(t,\xi)+(\Psi(t,\xi),x)}.
    \end{equation}
    A Markov process $X_t$ with such a characteristic function is  called an \emph{affine process}. Such processes have been recently considered in mathematical finance, cf.\ \cite{DFS}.

    If an affine process is regular---i.e.\ $q(x,\xi):=\partial_t \lambda_t(x,\xi)\big|_{t=0}$ exists for all $x$ and all $\xi\in \{z=(z_1,\ldots, z_m)\in \comp^m \::\: \Im z_j\geq 0, j=1,\ldots ,m\}\times \rn$ and is continuous at $\xi=0$---, then $X_t$ is a Feller process, see \cite[\S 8]{DFS}, hence \eqref{B1} is satisfied. Since for affine processes we explicitly know the representation of the characteristic function, it is easy to find conditions in terms of $\Phi$ and $\Psi$ such that \eqref{B2}--\eqref{B4} hold. For example, \eqref{B2}--\eqref{B4} are satisfied for $\Phi=\Psi^Y=0$, $\Psi^Z(\xi)=(\psi_{t,1}(\xi), \ldots,\psi_{t,n}(\xi))$ where $\psi_{t,j}(\xi)$, $j=1\ldots n$, are continuous negative definite functions satisfying the conditions of Theorem~\ref{pr} for all values of the parameter $t$.
\end{example}

Note that \eqref{B33} is equivalent to the finiteness of the following integral:
\begin{equation}\label{om}
    w_t(x,\zeta):=\ln \Big[e^{x\cdot\zeta} \int_{\rn} e^{-\zeta\cdot y} p_t(x,y)\, dy \Big]=\ln \lambda_t(x,i\zeta) \quad \text{for all\ \ } \zeta\in\rn,
\end{equation}
which is  a convex function of $\zeta$. Indeed, let $0<\alpha<1$, $\zeta,\xi\in\rn$. Then
\begin{align*}
    \alpha \ln &\int_{\rn} e^{\zeta\cdot y}p_t(x,y)\, dy + (1-\alpha) \ln \int_{\rn} e^{\xi\cdot y} \,p_t(x,y)\, dy \\
    &=\ln \left[ \left(\int_{\rn} e^{\zeta\cdot y}p_t(x,y)\, dy \right)^\alpha \left(\int_{\rn} e^{\xi\cdot y}\,p_t(x,y)\, dy\right)^{1- \alpha} \right]\\
    &\geq \ln \int_{\rn} e^{\alpha \zeta\cdot y +(1-\alpha)\xi\cdot y}\,p_t(x,y)\, dy.
\end{align*}
Since $\nabla_\zeta  \ln \int_{\rn} e^{\zeta\cdot y} \,p_t(x,y)\, dy\Big|_{\zeta=0}=0$, the function
\begin{equation}\label{vz}
    v_t(x-y,x,\zeta)
    :=-\zeta\cdot(x-y)+\ln \lambda_t(x,i\zeta)= -\zeta\cdot y + \ln \int_{\rn} e^{-\zeta\cdot h} \, p_t(x,h)\,dh
\end{equation}
has a minimum, see \cite[Proposition 47.12]{Zei}.
Define
\begin{equation}\label{z0}
    \zeta_0 =\zeta_0(t,x,y):=\arg\min_\zeta v_t(x-y,x,\zeta).
\end{equation}
By construction, there exists the analytic extension of $v_t(x-y,x,\cdot)$ to $\comp^n$. This means that the arguments of Section~\ref{sec2} can be used to show the next theorem.
\begin{theorem}\label{pr2}
    Let $q(x,D)$ be defined by \eqref{qD}, and suppose that \eqref{B1}--\eqref{B4} are satisfied. Then the transition density of the probability measure associated with $q(x,D)$ exists and satisfies
    \begin{equation}\label{pt-gen}
        p_t(x,y)
        \leq e^{v_t(x-y,x,\zeta_0)} \|\lambda_t(x,\cdot)\|_{L_1},
        \quad\text{for all}\quad x,y\in\rn,\; t>0,
    \end{equation}
    where $v_t(x-y,x,\zeta)$ and $\zeta_0$ are defined by \eqref{vz} and \eqref{z0}, respectively.
\end{theorem}
\begin{proof}
    Since
    \begin{equation}
        p_t(x,y)= \int_{\rn} e^{i\xi\cdot (x-y)}\lambda_t(x,\xi)\, d\xi,\label{ptnew}
    \end{equation}
    we can use, under our assumptions on $\lambda_t(x,\xi)$, the same approach as in Section~\ref{sec2}. By Cauchy's theorem we get
    \begin{equation}\label{ptnew2}
        p_t(x,y)
        =\int_{\rn} e^{-\zeta_0\cdot(x-y)+i\eta\cdot(x-y)}\lambda_t(x,\eta+i\zeta_0)\,d\eta.
    \end{equation}
    Hence, by condition \eqref{B4},
    \begin{align*}
        p_t(x,y)
        &\leq e^{v_t(x-y,x,\zeta_0)} \int_{\rn} \left|\frac{\lambda_t(x,\eta+i\zeta_0)}{\lambda_t(x,i\zeta)}\right|d\eta\\
        &\leq e^{v_t(x-y,x,\zeta_0)} \int_{\rn} |\lambda_t(x,\eta)|\,d\eta,
     \end{align*}
    which proves \eqref{pt-gen}.
\end{proof}

We can get an upper bound for the transition density $p_t(x,y)$ in terms of the symbol $q(x,\xi)$ and some  remainder term. For this we need further assumptions, e.g.\ that the symbol of the generator belongs to a certain symbol class introduced in \cite{Hoh98}, see also \cite[Definitions 2.4.3 and 2.4.4]{J02}.

To state the corollary of Theorem~\ref{pr2} we need a few basic facts of symbol classes. Let $\rho(|\alpha|):= |\alpha|\wedge 2$.
\begin{definition}
\begin{enumerate}
\item[\upshape i)]
    A continuous negative definite function $\psi:\rn\to\real$ belongs to the class $\Lambda$ if for all $\alpha\in\mathbb{N}_0^n$ there exists a constant $c_{|\alpha|}\geq 0$ such that
    \begin{equation}
        \left| \partial_\xi^\alpha (1+\psi(\xi))\right|
        \leq c_{|\alpha|} (1+\psi(\xi))^{\frac{2-\rho(|\alpha|)}{2}}.
    \end{equation}
    holds for all $\xi\in\rn$
\item[\upshape ii)]
    Let $m\in\real$, $\psi\in\Lambda$. A $C^\infty$-function $q:\rn\times\rn\to\comp$ is a symbol in the class $S_{\rho}^{m,\psi}(\rn) $ if for all $\alpha,\beta\in \mathbb{N}_0^n$ there are constants $c_{\alpha\beta}\geq 0$ such that
    \begin{equation}
        \left| \partial_\xi^\alpha \partial_x^\beta q(x,\xi)\right|
        \leq c_{\alpha\beta} (1+\psi(\xi))^{\frac{m-\rho(|\alpha|)}{2}}
    \end{equation}
    holds for all $\xi,x\in\rn$.
\end{enumerate}
\end{definition}

If $\rho\equiv 0$  we will simply write $S_0^{m,\psi}(\rn)$. From now on we will also assume that the symbol $q\in S_{\rho}^{2,\psi}(\rn)$ satisfies
\begin{gather}\tag{\bfseries B5}\label{B5}
    q(x,\xi) \geq c_r\,(1+\psi(\xi))
    \quad\text{for $x\in\rn$ and sufficiently large}\quad |\xi|\geq r.
\end{gather}
where $\psi\in\Lambda$ and $\psi(\xi)\geq c_1|\xi|^\kappa$ for some $\kappa>0$.

Note that \eqref{B5} implies \eqref{B1}, i.e.\ $(q(x,D),C_0^\infty(\rn))$ extends to the generator of a Feller semigroup, see \cite[Theorem~2.6.9]{J02}. By Theorem 2.8 from \cite{BB05} we can decompose $\lambda_t(x,\xi)$ in the following way.
\begin{equation}
    \lambda_t(x,\xi)=e^{-tq(x,\xi)}+r(t,x,\xi),
\end{equation}
where $r(t,x,\xi)\to 0$ weakly in $S_0^{-1,\psi}(\rn)$ as $t\to 0$.

Assume that \eqref{B2}, \eqref{B3} and \eqref{B5} hold. Then $r(t,x,\xi) \in L_1(\rn)$, and by \eqref{plam}
\begin{equation}
    p_t(x,y)
    =\int_{\rn} e^{i\xi\cdot(x-y)-tq(x,\xi)}\,d\xi + \int_{\rn} e^{i\xi\cdot(x-y)} r(t,x,\xi)\,d\xi.
\end{equation}
This allows us to formulate the following corollary of Theorem~\ref{pr2}.
\begin{corollary}
    Let $q\in S_{\rho}^{2,\psi}(\rn)$, and suppose that \eqref{B2}, \eqref{B3} and \eqref{B5} are satisfied. Then the transition function of the Feller process generated by $-q(x,D)$ has a density $p_t(x,y)$, and for all $x,y\in \rn$, $t>0$,
    \begin{equation}
        p_t(x,y)
        \leq e^{v_t(x-y,x,\zeta_0)} \int_{\rn} e^{-tq(x,\xi)}\,d\xi+ \int_{\rn} e^{i\xi\cdot(x-y)} \,r(t,x,\xi)\,d\xi,
     \end{equation}
    where $v_t(x-y,x,\zeta):=-\zeta\cdot(x-y)-tq(x,i\zeta)$, and $\zeta_0:=\arg\min_\zeta v_t(x-y,x,\zeta)$.
\end{corollary}

\begin{ack}
    Financial support through INTAS (grant YSF 06-1000019-6024),
    DAAD (research grant June-August 2009), and the Ministry of Science of Ukraine (grant no.\ M/7-2008)
    (for V.K.) and DFG (grant Schi 419/5-1) (for R.L.S.) is gratefully acknowledged.
     We would like to thank two anonymous referees for their careful reading; their comments
    helped to improve the quality of this paper.
\end{ack}

\end{document}